\documentclass[11pt]{amsart}
\usepackage{amsmath,amsthm,amsfonts,amssymb,mathrsfs, mathtools}
\usepackage{enumerate}
\usepackage[colorlinks=true, linkcolor=blue, citecolor=magenta, menucolor=black]{hyperref}
\usepackage[demo]{graphicx}
\usepackage[up]{caption}
\usepackage{subcaption}
\usepackage{pgf,tikz}
\usepackage[toc,page]{appendix}
\usetikzlibrary{arrows}
\usetikzlibrary{patterns}
\usepackage{color}
\usepackage{enumitem}

\theoremstyle{plain}
\newtheorem{theorem}{Theorem}[section]

\newtheorem{proposition}[theorem]{Proposition}

\newtheorem{letterthm}{Theorem}

\newtheorem{lettercor}[letterthm]{Corollary}

\theoremstyle{definition}
\newtheorem{definition}[theorem]{Definition}
\newtheorem{example}[theorem]{Example}
\newtheorem{remark}[theorem]{Remark}
\newtheorem{examples}[theorem]{Examples}
\newtheorem*{terminology}{Terminology}
\newtheorem*{Remark}{Remark}
\newtheorem*{Examples}{Examples}
\newtheorem*{connes}{Connes'\! rigidity conjecture}
\newtheorem{problem}[theorem]{Problem}
\newtheorem*{conjecture}{Conjecture}

\numberwithin{equation}{section}

\newcommand{\bfC}{\mathbf{C}}

\newcommand{\bfN}{\mathbf{N}}

\newcommand{\bfQ}{\mathbf{Q}}
\newcommand{\bfR}{\mathbf{R}}

\newcommand{\bfT}{\mathbf{T}}
\newcommand{\bfZ}{\mathbf{Z}}

\newcommand{\Stab}{\operatorname{Stab}}
\newcommand{\Fix}{\operatorname{Fix}}

\newcommand{\ovt}{\mathbin{\overline{\otimes}}}

\newcommand{\Ad}{\operatorname{Ad}}

\newcommand{\EL}{\operatorname{EL}}

\newcommand{\SL}{\operatorname{SL}}
\newcommand{\SO}{\operatorname{SO}}
\newcommand{\Ind}{\operatorname{Ind}}

\newcommand{\Prob}{\operatorname{Prob}}
\newcommand{\supp}{\operatorname{supp}}

\newcommand{\rk}{\operatorname{rk}}
\newcommand{\bary}{\operatorname{Bar}}

\newcommand{\Sp}{\operatorname{Sp}}

\newcommand{\Out}{\operatorname{Out}}

\newcommand{\dpr}{^{\prime\prime}}

\newcommand{\rE}{\operatorname{ E}}
\newcommand{\rC}{\operatorname{C}}
\newcommand{\rW}{\operatorname{W}}
\newcommand{\rL}{\operatorname{ L}}

\newcommand{\Ball}{\operatorname{Ball}}
\newcommand{\Har}{\operatorname{Har}}

\newcommand{\PSL}{\operatorname{PSL}}

\newcommand{\Char}{\operatorname{Char}}

\allowdisplaybreaks

\begin{document}



\title[Noncommutative ergodic theory of higher rank lattices]{Noncommutative ergodic theory \\ of higher rank lattices}


\author{Cyril Houdayer}
\address{Universit\'e Paris-Saclay \\ Institut Universitaire de France \\  Laboratoire de Math\'ematiques d'Orsay\\ CNRS \\ 91405 Orsay\\ FRANCE}
\email{cyril.houdayer@universite-paris-saclay.fr}
\thanks{The author is supported by Institut Universitaire de France}




\begin{abstract}
We survey recent results regarding the study of dynamical properties of the space of positive definite functions and characters of higher rank lattices. These results have several applications to ergodic theory, topological dynamics, unitary representation theory and operator algebras. The key novelty in our work is a dynamical dichotomy theorem for equivariant faithful normal unital completely positive maps between noncommutative von Neumann algebras and the space of bounded measurable functions defined on the Poisson boundary of semisimple Lie groups.
\end{abstract}

\maketitle


\section{Introduction and main results}\label{section:introduction}

In order to explain the motivation for our work and to state our main results, we set up the following terminology regarding \emph{higher rank lattices}.

\begin{terminology}
Let  $G$ be any connected semisimple real Lie group with finite center, no nontrivial compact factor and real rank $\rk_{\bfR}(G) \geq 2$. Let $\Gamma < G$ be any \emph{irreducible lattice}, meaning that $\Gamma$ is a discrete subgroup of $G$ with finite covolume such that $N \cdot \Gamma$ is a dense subgroup of $G$ for every noncentral closed normal subgroup $N \lhd G$. In what follows, if all the above conditions are satisfied, then we simply say that $\Gamma < G$ is a \emph{higher rank lattice}.
\end{terminology}

The following examples of higher rank lattices are particular cases of general results due to Borel--Harish-Chandra \cite{BHC61}.

\begin{Examples}
For every $d \geq 2$, the special linear group $\SL_d(\bfR)$ is a  connected simple real Lie group with finite center $\mathscr Z(\SL_d(\bfR)) = \{\pm 1_d\}$ and real rank $\rk_{\bfR}(\SL_d(\bfR))  = d - 1$.
\begin{enumerate}
\item For every $d \geq 3$, $\SL_d(\bfZ) < \SL_d(\bfR)$ is a higher rank lattice.
\item For every $d \geq 2$ and every square free integer $q \in \bfN$, 
\begin{equation*}
\Gamma \coloneqq \{ (g, g^\sigma) \mid g \in \SL_d(\bfZ[\sqrt{q}]) \} < \SL_d(\bfR) \times \SL_d(\bfR) \coloneqq G
\end{equation*}
is a higher rank lattice, where $\sigma$ is the order $2$ automorphism of $\bfQ(\sqrt{q})$.
\end{enumerate}
\end{Examples}

The main inspiration for our work is Margulis'\! celebrated \emph{normal subgroup theorem} which states that for any higher rank lattice $\Gamma < G$, any normal subgroup $N \lhd \Gamma$ is either finite and contained in $\mathscr Z(\Gamma)$ or $N$ has finite index in $\Gamma$ (see \cite[Theorem IV.4.9]{Ma91}). Margulis'\! remarkable strategy to prove the normal subgroup theorem consists of two ``halves'': the \emph{amenability half} and the \emph{property} (T) \emph{half}. Indeed, assuming that $N \lhd \Gamma$ is a noncentral normal subgroup, to prove that the quotient group $\Gamma/ N$ is finite, Margulis showed that $\Gamma/N$ is both amenable and has property (T). The proof of the amenability half relies on Margulis'\! \emph{factor theorem} which states that any measurable $\Gamma$-factor of the homogeneous space $G/P$, where $P < G$ is a minimal parabolic subgroup, is measurably isomorphic to a $G$-factor whence of the form $G/Q$, where $P < Q < G$ is an intermediate parabolic subgroup (see \cite[Theorem IV.2.11]{Ma91}). Margulis'\! strategy has been used to prove a normal subgroup theorem for various classes of irreducible lattices in product groups (see \cite{BM00, Sh99, BS04, CS13}) and to understand the structure of point stabilizers of ergodic probability measure preserving actions of higher rank lattices (see \cite{SZ92, CP12}). More recently, Margulis'\! strategy has been adapted to the noncommutative setting to study \emph{characters} of higher rank lattices (see \cite{CP13, Pe14}).

In that respect, for any countable discrete group $\Lambda$, we denote by $\mathscr P(\Lambda)$ the space of positive definite functions $\varphi : \Lambda \to \bfC$ normalized so that $\varphi(e) = 1$. Then $\mathscr P(\Lambda) \subset \ell^\infty(\Lambda)$ is a weak-$\ast$ compact convex subset. Thanks to the Gelfand--Naimark--Segal (GNS) construction, to any positive definite function $\varphi \in \mathscr P(\Lambda)$ corresponds a triple $(\pi_\varphi, \mathscr H_\varphi, \xi_\varphi)$,  where $\pi_\varphi : \Lambda \to \mathscr U(\mathscr H_\varphi)$ is a unitary representation and $\xi_\varphi \in \mathscr H_\varphi$ is a unit vector such that the linear span of $\pi_\varphi(\Lambda)\xi_\varphi$ is dense in $\mathscr H_\varphi$ and 
\begin{equation*}
\forall \gamma \in \Lambda, \quad \varphi (\gamma) = \langle \pi_\varphi(\gamma)\xi_\varphi, \xi_\varphi\rangle.
\end{equation*}
We consider the conjugation action $\Lambda \curvearrowright \mathscr P(\Lambda)$ defined by 
\begin{equation*}
\forall \gamma, g \in \Lambda, \forall \varphi \in \mathscr P(\Lambda), \quad (\gamma \varphi)(g) \coloneqq \varphi(\gamma^{-1} g \gamma).
\end{equation*}
A fixed point $\varphi \in \mathscr P(\Lambda)$ for the conjugation action is called a \emph{character}. We denote by $\Char(\Lambda) \subset \mathscr P(\Lambda)$ the weak-$\ast$ compact convex subset of all characters. Any countable discrete group $\Lambda$ always admits at least two characters: the \emph{trivial} character $1_\Lambda$ and the \emph{regular} character $\delta_e$. The GNS representation of the regular character $\delta_e$ coincides with the left regular representation $\lambda : \Lambda \to \mathscr U(\ell^2(\Lambda))$. An important source of characters comes from ergodic theory. Indeed, for any probability measure preserving action $\Lambda \curvearrowright (X, \nu)$ on a standard probability space, the function $\varphi : \Lambda \to \bfC : \gamma \mapsto \nu(\Fix(\gamma))$ defines a character. The action $\Lambda \curvearrowright (X, \nu)$ is (essentially) free if and only if the above character $\varphi$ is equal to $\delta_e$.

For any unitary representation $\pi : \Lambda \to \mathscr U(\mathscr H_\pi)$, we consider the unital $\rC^*$-algebra 
$$\rC^*_\pi(\Lambda) \coloneqq \rC^*(\{\pi(\gamma) \mid \gamma \in \Lambda \}) \subset \mathrm B(\mathscr H_\pi)$$ endowed with the conjugation action $\Ad(\pi) : \Lambda \curvearrowright \rC^*_\pi(\Lambda)$. We then regard the state space $\mathfrak S(\rC^*_\pi(\Lambda))$ as a $\Lambda$-invariant weak-$\ast$ compact convex subset of $\mathscr P(\Lambda)$ via the mapping $\mathfrak S(\rC^*_\pi(\Lambda)) \hookrightarrow \mathscr P(\Lambda) : \psi \mapsto \psi \circ \pi$. When $\pi = \lambda$ is the left regular representation, $\rC^*_\lambda(\Lambda)$ is the \emph{reduced} group $\rC^*$-algebra which is endowed with the canonical faithful trace $\tau_\Lambda$ defined by $\tau_\Lambda : \rC^*_\lambda(\Lambda) \to \bfC : a \mapsto \langle a \delta_e, \delta_e\rangle$. 

Given unitary representations $\pi_i : \Lambda \to \mathscr U(\mathscr H_i)$, $i = 1, 2$, we say that $\pi_2$ is \emph{weakly contained} in $\pi_1$ if the map $\pi_1(\Lambda) \to \pi_2(\Lambda) : \pi_1(\gamma) \mapsto \pi_2(\gamma)$ is well defined and extends to a $\ast$-homomorphism $\rC^*_{\pi_1}(\Lambda) \to \rC^*_{\pi_2}(\Lambda)$. Following \cite{Be89}, we say that a unitary representation $\pi : \Lambda \to \mathscr U(\mathscr H_\pi)$ is \emph{amenable} if the trivial representation $1_\Lambda$ is weakly contained in $\pi \otimes \overline \pi$. If $\pi$ contains a finite dimensional subrepresentation, then $\pi$ is amenable. If $\Lambda$ has property (T), then conversely any amenable representation $\pi : \Lambda \to \mathscr U(\mathscr H_\pi)$ contains a finite dimensional subrepresentation.

We now present in a unified way the main results we obtained in \cite{BH19} (joint work with R.\ Boutonnet) and \cite{BBHP20} (joint work with U.\ Bader, R.\ Boutonnet and J.\ Peterson). Our first main result deals with the \emph{existence of characters}. It is a fixed point theorem for the affine action of higher rank lattices on their space of positive definite functions.

\begin{letterthm}[\cite{BH19, BBHP20}]\label{thm:main}
Let $\Gamma < G$ be any higher rank lattice. Then any nonempty $\Gamma$-invariant weak-$\ast$ compact convex subset $\mathscr C \subset \mathscr P(\Gamma)$ contains a character.
\end{letterthm}

Our second main result deals with the \emph{classification of characters} of higher rank lattices. Bekka \cite{Be06} obtained the first character rigidity results in the case $\Gamma = \SL_d(\bfZ)$ for $d \geq 3$.  More recently, using a different approach based on Margulis'\! strategy discussed above, Peterson \cite{Pe14} obtained character rigidity results for arbitrary higher rank lattices (see also \cite{CP13} for the case of irreducible lattices in certain product groups).  The operator algebraic framework we developed in \cite{BH19, BBHP20} enables us to obtain a new and more conceptual proof of Peterson's character rigidity results \cite{Pe14}.

\begin{letterthm}[Peterson, \cite{Pe14}]\label{thm:characters}
Let $\Gamma < G$ be any higher rank lattice. Then any character $\varphi \in \Char(\Gamma)$ is either supported on $\mathscr Z(\Gamma)$ or its GNS representation $\pi_\varphi$ is amenable.

In case $G$ has a simple factor with property \emph{(T)}, any character $\varphi \in \Char(\Gamma)$ is either supported on $\mathscr Z(\Gamma)$ or its GNS representation $\pi_\varphi$ contains a finite dimensional subrepresentation.
\end{letterthm}

Theorem \ref{thm:characters} generalizes Margulis'\! normal subgroup theorem \cite{Ma91} and Stuck--Zimmer's stabilizer rigidity theorem \cite{SZ92}. Also, Theorem \ref{thm:characters} solved a conjecture formulated by Connes (see \cite{Jo00}). For other recent results regarding classification of characters, we refer the reader to \cite{CP13, PT13, Be19, BF20, LL20}.

Combining Theorems \ref{thm:main} and \ref{thm:characters}, we obtain new results regarding the simplicity and the unique trace property for the $\rC^*$-algebra $\rC^*_\pi(\Gamma)$ associated with an arbitrary nonamenable (resp.\! weakly mixing) unitary representation $\pi : \Gamma \to \mathscr U(\mathscr H_\pi)$. In particular, Corollary \ref{cor:C*} provides a far reaching generalization of the results obtained by Bekka--Cowling--de la Harpe \cite{BCH94} for the reduced $\rC^*$-algebra $\rC^*_\lambda(\Gamma)$.

\begin{lettercor}[\cite{BH19, BBHP20}]\label{cor:C*}
Let $\Gamma < G$ be any higher rank lattice. Let $\pi : \Gamma \to \mathscr U(\mathscr H_\pi)$ be any unitary representation. Then $\rC^*_\pi(\Gamma)$ admits a trace. 

Assume moreover that $G$ has trivial center. If $\pi$ is not amenable, then $\lambda$ is weakly contained in $\pi$ and the unique $\ast$-homomorphism $\Theta : \rC^*_\pi(\Gamma) \to \rC^*_\lambda(\Gamma) : \pi(\gamma) \mapsto \lambda(\gamma)$ satisfies the following properties:
\begin{enumerate}
\item $\tau_\Gamma \circ \Theta$ is the unique trace on $\rC^*_\pi(\Gamma)$.
\item $\ker(\Theta)$ is the unique proper maximal ideal of $\rC^*_\pi(\Gamma)$.
\end{enumerate}
In case $G$ has property \emph{(T)}, the above properties hold as soon as $\pi$ does not contain any nonzero finite dimensional subrepresentation.
\end{lettercor} 

In case $\Lambda$ is a countable discrete group and $\pi = \lambda$, Hartman--Kalantar \cite{HK17} initiated the study of the noncommutative dynamical system $\Lambda \curvearrowright \rC^*_\lambda(\Lambda)$ and obtained a new characterization of the simplicity and the unique trace property for $\rC^*_\lambda(\Lambda)$ (see also \cite{KK14, BKKO14}).

As a byproduct of our operator algebraic methods, we also obtain a topological analogue of Stuck--Zimmer's stabilizer rigidity theorem \cite{SZ92}. In particular, our next result gives a positive answer to a recent problem raised by Glasner--Weiss \cite{GW14}.

\begin{letterthm}[\cite{BH19, BBHP20}]\label{thm:dynamics}
Let $\Gamma < G$ be any higher rank lattice and assume that $G$ has trivial center. Let $\Gamma \curvearrowright X$ be any minimal action on a compact space. Then at least one of the following assertions hold:
\begin{enumerate}
\item There exists a $\Gamma$-invariant Borel probability measure on $X$.

\item The action $\Gamma \curvearrowright X$ is topologically free.
\end{enumerate}
In case $G$ has a simple factor with property \emph{(T)}, either $X$ is finite or the action $\Gamma \curvearrowright X$ is topologically free.
\end{letterthm}

As we will explain, all the main results stated above are consequences of a dynamical dichotomy theorem for $\Gamma$-equivariant faithful normal unital completely positive (ucp) maps $\Phi : M \to \rL^\infty(G/P)$, where $M$ is an arbitrary von Neumann algebra endowed with an ergodic action $\Gamma \curvearrowright M$ (see Theorem \ref{thm:NCNZ} and Theorem \ref{thm:dynamical} below). This dynamical dichotomy theorem is one of the key novelties of our operator algebraic framework. Both its statement and its proof rely on von Neumann algebra theory and depend heavily on whether the connected semisimple real Lie group $G$ is simple or not. 

In our first joint work \cite{BH19}, we dealt with the case where $G$ is simple and $\rk_{\bfR}(G) \geq 2$ (and more generally where all its simple factors $G_i$ satisfy $\rk_{\bfR}(G_i) \geq 2$). In that case, we obtained the following noncommutative analogue of Nevo--Zimmer's structure theorem \cite{NZ97, NZ00}. We denote by $P < G$ a minimal parabolic subgroup and whenever $P < Q < G$ is an intermediate parabolic subgroup, we denote by $p_Q : G/P \to G/Q : gP \mapsto gQ$ the canonical factor map and by $p_Q^\ast : \rL^\infty(G/Q) \to \rL^\infty(G/P) : f \mapsto f \circ p_Q$ the corresponding unital normal embedding.

\begin{letterthm}[\cite{BH19}]\label{thm:NCNZ}
Let $\Gamma < G$ be any higher rank lattice and assume that $G$ is simple. Let $M$ be any von Neumann algebra, $\Gamma \curvearrowright M$ any ergodic action and $\Phi : M \to \rL^\infty(G/P)$ any $\Gamma$-equivariant faithful normal ucp map. Then the following dichotomy holds:
\begin{itemize}
\item Either $\Phi(M) = \bfC 1$.
\item Or there exist a proper parabolic subgroup $P < Q < G$ and a $\Gamma$-equivariant unital normal embedding $\iota : \rL^\infty(G/Q) \hookrightarrow M$ such that $\Phi \circ \iota = p_Q^\ast$. 
\end{itemize}
\end{letterthm}

Theorem \ref{thm:NCNZ} extends the work of Nevo--Zimmer in two ways. Firstly, we deal with arbitrary von Neumann algebras $M$ (instead of measure spaces $(X, \nu)$) and secondly, we deal with $\Gamma$-actions $\Gamma \curvearrowright M$ (instead of $G$-actions $G \curvearrowright (X, \nu)$). We refer to Section \ref{section:structures} for the correspondence between equivariant ucp maps and stationary states. The remarkable feature of Theorem \ref{thm:NCNZ} is that when $\Phi : M \to \rL^\infty(G/P)$ is not invariant, there is a nontrivial $\Gamma$-invariant \emph{commutative} von Neumann subalgebra $M_0 \subset M$ such that $M_0 \cong \rL^\infty(G/Q)$. This allows us to exploit the dynamical properties of the ergodic action $\Gamma \curvearrowright G/Q$ and the fact that every noncentral element $\gamma \in \Gamma \setminus \mathscr Z(\Gamma)$ acts (essentially) freely on $G/Q$. In particular in \cite{BH19}, we used Theorem \ref{thm:NCNZ} to derive all the main results stated above.

In our second joint work \cite{BBHP20}, we dealt with the case where $G$ is not simple. We point out that when $G$ has a rank one simple factor (e.g.\ $G = \SL_2(\bfR) \times \SL_2(\bfR)$), a Nevo--Zimmer structure theorem does not hold and the method used in \cite{BH19} to prove the main results does not apply. Instead, we proceeded as follows \cite{BBHP20}. Firstly, for \emph{any} higher rank lattice $\Gamma < G$, we formulated a general dynamical dichotomy theorem \emph{invariant vs.\ singular} for $\Gamma$-equivariant faithful normal ucp maps $\Phi : M \to \rL^\infty(G/P)$ (see Theorem \ref{thm:dynamical} below) and we showed that all the main results stated above can be derived from this general dichotomy theorem. In case $G$ is simple, the dynamical dichotomy Theorem \ref{thm:dynamical} is a straightforward consequence of Theorem \ref{thm:NCNZ}. Secondly, to prove the dynamical dichotomy Theorem \ref{thm:dynamical} in the case where $G$ is not simple, we developed a new method based on the product structure in $G$. In that respect, the tools developed in \cite{BH19} and \cite{BBHP20} are complementary.

For any nonsingular action $\Lambda \curvearrowright (X, \nu)$ on a standard probability space, we denote by $\rL(\Lambda \curvearrowright X)$ the corresponding \emph{group measure space} von Neumann algebra (see Section \ref{section:preliminaries}). For higher rank lattices $\Gamma < G$ where $G$ is simple with trivial center, we present yet another application of Theorem \ref{thm:NCNZ} that will appear in \cite{BBH21} (joint work with U.\ Bader and R.\ Boutonnet). The next result can be regarded as a noncommutative analogue of Margulis'\! factor theorem.

\begin{lettercor}[\cite{BBH21}]\label{thm:NCFT}
Let $\Gamma < G$ be any higher rank lattice and assume that $G$ is simple with trivial center. Let  $\rL(\Gamma) \subset M \subset \rL(\Gamma \curvearrowright G/P)$ be any intermediate von Neumann subalgebra. Then there is a unique intermediate parabolic subgroup $P < Q < G$ such that 
\begin{equation*}
M = \rL(\Gamma \curvearrowright G/Q).
\end{equation*}
\end{lettercor}

In Section \ref{section:preliminaries}, we give some preliminary background on Poisson boundaries, semisimple Lie groups and operator algebras. In Section \ref{section:structures}, we introduce the notion of boundary structures for von Neumann algebras and state the dynamical dichotomy theorem. We also outline the main steps of the proof of Theorem \ref{thm:NCNZ}. In Section \ref{section:main}, we sketch the proofs of Theorems \ref{thm:main}, \ref{thm:characters} and Corollary \ref{cor:C*} based on the dynamical dichotomy theorem. We also discuss open problems related to our main results. In Section \ref{section:NCFT}, we discuss Corollary \ref{thm:NCFT} and its relevance for Connes'\! rigidity conjecture.

\begin{Remark}
In this survey article, we only consider higher rank lattices in connected semisimple \emph{real}  Lie groups to simplify the exposition and to focus on the main ideas. However, we point out that all our main results do hold for higher rank lattices in semisimple algebraic groups defined over arbitrary local fields. We refer the reader to \cite{BBHP20, BBH21} for more general statements and further details.
\end{Remark}

\subsection*{Acknowledgments}
I would like to express my deepest gratitude to my co-authors for our exciting joint works. I also thank Adrian Ioana for his useful remarks.

\section{Preliminaries}\label{section:preliminaries}

\subsection{Poisson boundaries}\label{subsection:poisson}

Let $H$ be any locally compact second countable (lcsc) group. We say that a Borel probability measure $\mu \in \Prob(H)$ is \emph{admissible} if the following conditions are satisfied:
\begin{enumerate}
\item $\mu$ is absolutely continuous with respect to the Haar measure;

\item $\supp(\mu)$ generates $H$ as a semigroup;

\item $\supp(\mu)$ contains a neighborhood of the identity element $e \in H$.
\end{enumerate}
We say that a bounded measurable function $F : H \to \bfC$ is (right) $\mu$-\emph{harmonic} if 
\begin{equation*}
\forall g \in H, \quad F(g) = \int_H F(gh) \, \mathrm{d}\mu(h).
\end{equation*}
Any $\mu$-harmonic function is continuous. We denote by $\Har^\infty(H, \mu) \subset \rC_b(H)$ the space of all (right) $\mu$-harmonic functions. The left translation action $\lambda : H \curvearrowright \rC_b(H)$ leaves the subspace $\Har^\infty(H, \mu)$ globally invariant.

Let $(X, \nu)$ be any standard probability space endowed with a measurable action $H \curvearrowright X$. We say that $(X, \nu)$ is a $(H, \mu)$-\emph{space} if $\nu$ is $\mu$-\emph{stationary}, that is, $\mu \ast \nu = \nu$. For any $(H, \mu)$-space $(X, \nu)$, define the \emph{Poisson transform} $\Psi_\mu : \rL^\infty(X, \nu) \to \Har^\infty(H, \mu)$ by the formula
\begin{equation*}
\forall f \in \rL^\infty(X, \nu), \forall g \in H, \quad \Psi_\mu(f)(g) = \int_X f(g x) \, \mathrm{d}\nu(x).
\end{equation*}
The mapping $\Psi_\mu : \rL^\infty(X, \nu) \to \Har^\infty(H, \mu)$ is $H$-equivariant, unital, positive and contractive.  

\begin{theorem}[Furstenberg, \cite{Fu62b}]
There is a unique $(H, \mu)$-space $(B, \nu_B)$ for which the Poisson transform $\Psi_\mu : \rL^\infty(B, \nu_B) \to \Har^\infty(H, \mu)$ is bijective.
\end{theorem}

The $(H, \mu)$-space $(B, \nu_B)$ is called the $(H, \mu)$-\emph{Poisson boundary}. For a construction of the $(H, \mu)$-space $(B, \nu_B)$, we also refer to \cite{BS04, Fu00}. The $(H, \mu)$-space $(B, \nu_B)$ enjoys remarkable ergodic theoretic properties. In that respect, let $(E, \|\cdot\| )$ be any separable continuous isometric Banach $H$-module and $\mathscr C \subset E^*$ any nonempty $H$-invariant weak-$\ast$ compact convex subset. Denote by $\bary : \Prob(\mathscr C) \to \mathscr C$ the $H$-equivariant continuous barycenter map. A point $c \in \mathscr C$ is $\mu$-\emph{stationary} if $\bary({\iota_c}_\ast\mu) = c$ where $\iota_c : H \to \mathscr C : g \mapsto gc$ is the orbit map associated with $c \in \mathscr C$. By Markov--Kakutani's fixed point theorem, the subset $\mathscr C_\mu \subset \mathscr C$ of all $\mu$-stationary points in $\mathscr C$ is not empty. 

The following theorem due to Furstenberg provides the existence (and uniqueness) of boundary maps (see also \cite[Section 2]{BS04}).

\begin{theorem}[Furstenberg, \cite{Fu62b}]
Let $c \in \mathscr C_\mu$ be any $\mu$-stationary point. Then there exists an (essentially) unique $H$-equivariant measurable map $\beta : B \to \mathscr C$ such that 
\begin{equation*}
\bary({\beta}_\ast \nu_B ) = c.
\end{equation*}
We say that $\beta : B \to \mathscr C$ is the $H$-equivariant \emph{boundary map} associated with $c \in \mathscr C_\mu$.
\end{theorem}

\subsection{Semisimple Lie groups}

Let $G$ be any connected semisimple real Lie group with finite center and no nontrivial compact factors. Fix an Iwasawa decomposition $G = KAV$, where $K < G$ is a maximal compact subgroup, $A < G$ is a Cartan subgroup and $V < G$ is a unipotent subgroup. Denote by $L  \coloneqq \mathscr Z_G(A)$ the centralizer of $A$ in $G$ and set $P \coloneqq L V$. Then $P < G$ is a minimal parabolic subgroup. Since $K \curvearrowright G/P$ is transitive, $G/P$ is a compact homogeneous space and there exists a unique $K$-invariant Borel probability measure $\nu_P \in \Prob(G/P)$. The measure class of $\nu_P$ coincides with the unique $G$-invariant measure class on $G/P$. 

\begin{example}
Assume that $G = \SL_d(\bfR)$ for $d \geq 2$. Then we may take $K = \SO_d(\bfR)$, $A < G$ the subgroup of diagonal matrices and $V < G$ the subgroup of strict upper triangular matrices. In that case, $P = AV < G$ is the subgroup of upper triangular matrices. The homogeneous space $G/P$ is the \emph{full flag variety} which consists in all flags $\{0\} \subset W_1 \subset \cdots \subset W_d = \bfR^d$, where $W_i \subset \bfR^d$ is a vector subspace such that $\dim_{\bfR}(W_i) = i$ for every $1 \leq i \leq d$.
\end{example}

Observe for any left $K$-invariant Borel probability measure $\mu_G \in \Prob(G)$, the probability measure $\mu_G \ast \nu_P$ is $K$-invariant on $G/P$ and so $\mu_G \ast \nu_P = \nu_P$, that is, $(G/P, \nu_P)$ is a $(G, \mu_G)$-space. Furstenberg \cite{Fu62a} proved the following fundamental result describing the Poisson boundary of semisimple Lie groups.

\begin{theorem}[Furstenberg, \cite{Fu62a}]\label{thm:poisson-G}
Let $\mu_G \in \Prob(G)$ be any $K$-invariant admissible Borel probability measure. Then $(G/P, \nu_P)$ is the $(G, \mu_G)$-Poisson boundary.
\end{theorem}

For lattices $\Gamma < G$ in connected semisimple \emph{real}  Lie groups as above, Furstenberg \cite{Fu67} also showed that $(G/P, \nu_P)$ can be regarded as the $(\Gamma, \mu_\Gamma)$-Poisson boundary with respect to a well chosen probability measure $\mu_\Gamma \in \Prob(\Gamma)$ (see also \cite{Fu00} and the references therein).

\begin{theorem}[Furstenberg, \cite{Fu67}]\label{thm:poisson-Gamma}
Let $\Gamma < G$ be any lattice. Then there exists a probability measure $\mu_\Gamma \in \Prob(\Gamma)$ with full support such that $(G/P, \nu_P)$ is the $(\Gamma, \mu_\Gamma)$-Poisson boundary.
\end{theorem}

We call a probability measure $\mu_\Gamma \in \Prob (\Gamma)$ as in Theorem \ref{thm:poisson-Gamma} a \emph{Furstenberg measure}. Combining Theorems \ref{thm:poisson-G} and \ref{thm:poisson-Gamma}, we have
\begin{equation*}
\Har^\infty(G, \mu_G) \quad \underset{G\text{-equiv.}}{\cong} \quad \rL^\infty(G/P, \nu_P) \quad \underset{\Gamma\text{-equiv.}}{\cong} \quad \Har^\infty(\Gamma, \mu_\Gamma).
\end{equation*}

A combination of Theorem \ref{thm:poisson-Gamma} and \cite{GM89} implies that for any intermediate parabolic subgroup $P < Q < G$, the map 
\begin{equation}\label{eq:uniqueness}
\delta \circ p_Q : G/P \to \Prob(G/Q) : gP \mapsto \delta_{gQ}
\end{equation}
is the (essentially) unique $\Gamma$-equivariant measurable mapping $\zeta : G/P \to \Prob(G/Q)$.

\subsection{Operator algebras}

A $\rC^*$-\emph{algebra} $A$ is a Banach $\ast$-algebra endowed with a complete  norm $\|\cdot\|$ that satisfies the $\rC^*$-identity: $\|a^* a\| = \|a\|^2$, for every $a \in A$.
Any $\rC^*$-algebra $A$ admits a faithful isometric $\ast$-representation on a Hilbert space $\pi : A \to \mathrm B(\mathscr H)$. After identifying $A$ with $\pi(A)$, we may regard $A \subset \mathrm B(\mathscr H)$ as a concrete $\rC^*$-algebra. Unless stated otherwise, all $\rC^*$-algebras and all linear mappings between $\rC^*$-algebras are always assumed to be unital.

We denote by $\mathfrak S(A)$ the \emph{state space} of $A$. Then $\mathfrak S(A) \subset \Ball(A^*)$ is a weak-$\ast$ compact convex subset. We say that an action $\sigma : H \curvearrowright A$ is \emph{continuous} if the action map $H \times A \to A : (g, a) \mapsto \sigma_g(a)$ is continuous. We then simply say that $A$ is a $H$-$\rC^*$-algebra. The continuous action $H \curvearrowright A$ induces a weak-$\ast$ continuous affine action $H \curvearrowright \mathfrak S(A)$. We may apply the results from Subsection \ref{subsection:poisson} to the $H$-invariant weak-$\ast$ compact convex set $\mathscr C = \mathfrak S(A)$. When $\mu \in \Prob(H)$ is an admissible Borel probability measure, we denote by $\mathfrak S_\mu(A) \subset \mathfrak S(A)$ the nonempty weak-$\ast$ compact convex subset of all $\mu$-stationary states.

\begin{examples} We will consider the following examples of  $\rC^*$-algebras.
\begin{enumerate}
\item For any compact metrizable space $X$, the space $\rC(X)$ of all continuous functions on $X$ endowed with the uniform norm $\|\cdot\|_\infty$ is a commutative $\rC^*$-algebra. Any commutative $\rC^*$-algebra arises this way. We identify the set $\Prob(X)$ of Borel probability measures on $X$ with the state space $\mathfrak S(\rC(X))$ via the continuous mapping $\Prob(X) \to \mathfrak S(\rC(X)) : \nu \mapsto \int_X \cdot \; \mathrm{d}\nu$. Any continuous action  by homeomorphisms $H \curvearrowright X$ naturally gives rise to a continuous action $H \curvearrowright \rC(X)$ in the above sense.
\item For any countable discrete group $\Lambda$ and any unitary representation $\pi : \Lambda \to \mathscr U(\mathscr H_\pi)$, define the $\rC^*$-algebra
\begin{equation*}
\rC^*_\pi(\Lambda) \coloneqq \rC^*(\left\{ \pi(\gamma) \mid \gamma \in \Lambda\right\}) \subset \mathrm B(\mathscr H_\pi)
\end{equation*}
and consider the conjugation action $\Ad(\pi) : \Lambda \curvearrowright \rC^*_\pi(\Lambda)$. The state space $\mathfrak S(\rC^*_\pi(\Lambda))$ is a $\Lambda$-invariant weak-$\ast$ compact convex subset of $\mathscr P(\Lambda)$ via the mapping $\mathfrak S(\rC^*_\pi(\Lambda)) \hookrightarrow \mathscr P(\Lambda) : \psi \mapsto \psi \circ \pi$. If $\pi = \lambda$ is the left regular representation, then $\rC^*_\lambda(\Lambda)$ is the \emph{reduced} group $\rC^*$-algebra. Moreover, the state $\tau_\Lambda : \rC^*_\lambda(\Lambda) \to \bfC: a \mapsto \langle a \delta_e, \delta_e\rangle$ is a faithful trace.
\end{enumerate}
\end{examples}

A \emph{von Neumann algebra} (or $\rW^*$-\emph{algebra}) $M$ is a unital $\rC^*$-algebra which admits a faithful unital $\ast$-representation $\pi : M \to \mathrm B(\mathscr H)$ such that $\pi(M) \subset \mathrm B(\mathscr H)$ is closed with respect to the weak (equivalently strong) operator topology. After identifying $M$ with $\pi(M)$, we may regard $M \subset \mathrm B(\mathscr H)$ as a concrete von Neumann algebra. By von Neumann's bicommutant theorem, a unital $\ast$-subalgebra $M \subset \mathrm B(\mathscr H)$ is a von Neumann algebra if and only if $M$ is equal to its own bicommutant $M\dpr$, that is, $M = M\dpr$.  There is a unique Banach space predual $M_\ast$ such that $M = (M_\ast)^*$. The ultraweak topology on $M$ coincides with the weak-$\ast$ topology arising from the identification $M = (M_\ast)^*$. A linear mapping between von Neumann algebras is \emph{normal} if it is continuous with respect to the ultraweak topology. We say that an action $\sigma : H \curvearrowright M$ is \emph{continuous} if the corresponding action map $H \times M_\ast \to M_\ast : (g, \varphi) \mapsto \varphi \circ \sigma_g^{-1}$ is continuous (see e.g.\ \cite[Proposition X.1.2]{Ta03a}). We then simply say that $M$ is a $H$-von Neumann algebra. The action $H \curvearrowright M$ is \emph{ergodic} if the fixed point von Neumann subalgebra $M^H = \left\{ x \in M \mid \forall g \in H, \sigma_h(x)  = x \right \}$ is trivial.

\begin{examples}\label{ex:vN} We will consider the following examples of von Neumann algebras.
\begin{enumerate}
\item For any standard probability space $(X, \nu)$, the space  $ \rL^\infty(X, \nu)$ of all  $\nu$-equivalence classes of (essentially) bounded measurable functions endowed with the (essential) uniform norm $\|\cdot\|_\infty$ is a commutative von Neumann algebra. Any commutative von Neumann algebra arises this way.  Any nonsingular action $H \curvearrowright (X, \nu)$ naturally gives rise to a continuous action $H \curvearrowright \rL^\infty(X, \nu)$ in the above sense. When no confusion is possible, we simply write $\rL^\infty(X) = \rL^\infty(X, \nu)$.
\item For any countable discrete group $\Lambda$ and any nonsingular action $\Lambda \curvearrowright (X, \nu)$ on a standard probability space, define the \emph{group measure space von Neumann algebra}
\begin{equation*}
\rL(\Lambda \curvearrowright X) \coloneqq \{ f \otimes 1, \pi(\gamma) \mid f \in \rL^\infty(X), \gamma \in \Lambda\}\dpr \subset \mathrm B(\rL^2(X, \nu) \otimes \ell^2(\Lambda))
\end{equation*}
where $\pi : \Lambda \to \mathscr U(\rL^2(X, \nu) \otimes \ell^2(\Lambda))$ is the unitary representation defined by
\begin{equation*}
\forall \xi \in \rL^2(X, \nu), \forall \gamma, h \in \Lambda, \quad \pi(\gamma)(\xi \otimes \delta_h) = \sqrt{\frac{\mathrm{d}(\nu\circ \gamma^{-1})}{\mathrm{d}\nu}} \, \xi \circ \gamma^{-1} \otimes \delta_{\gamma h}.
\end{equation*}
When $(X, \nu)$ is a singleton, the von Neumann algebra $\rL(\Lambda \curvearrowright X)$ coincides with the \emph{group von Neumann algebra} $\rL(\Lambda)$. When the action $\Lambda \curvearrowright (X, \nu)$ is (essentially) free and ergodic, the von Neumann algebra $\rL(\Lambda \curvearrowright X)$ is a factor whose type coincides with the type of the action (see e.g.\ \cite[Theorem XIII.1.7]{Ta03b}).
\end{enumerate}
\end{examples}

A von Neumann algebra $M \subset \mathrm B(\mathscr H)$ is \emph{amenable} if there exists a norm one projection $\rE : \mathrm B(\mathscr H) \to M$. By Connes'\! fundamental result \cite{Co75}, $M$ is amenable if and only if $M$ is \emph{approximately finite dimensional}, that is, there exists an increasing net of finite dimensional subalgebras $M_i \subset M$ such that $\bigvee_{i \in I} M_i = M$.

\section{Dynamical dichotomy for boundary structures}\label{section:structures}

\subsection{Boundary structures}

For any $\rC^*$-algebra $A\subset \mathrm B(\mathscr H)$ and any $n \geq 1$, $\mathrm M_n(A) \coloneqq \mathrm M_n(\bfC) \otimes A \subset \mathrm B(\mathscr H^{\oplus n})$ is naturally  a $\rC^*$-algebra. Let $A, B$ be any $\rC^*$-algebras. A linear map $\Phi : A \to B$ is said to be \emph{unital completely positive} (ucp) if $\Phi$ is unital and if for every $n \geq 1$, the linear map $\Phi^{(n)} : \mathrm M_n(A) \to \mathrm M_n(B) : [a_{ij}]_{ij} \mapsto [\Phi(a_{ij})]_{ij}$ is positive. Any unital $\ast$-homomorphism $\pi : A \to B$ is a ucp map. When $A$ or $B$ is commutative, any unital positive linear map $\Phi : A \to B$ is automatically ucp (see e.g.\ \cite[Theorems 3.9 and 3.11]{Pa02}).

\begin{definition}[\cite{BBHP20}]
Let $\Gamma < G$ be any higher rank lattice and $M$ any $\Gamma$-von Neumann algebra with separable predual. A $\Gamma$-\emph{boundary structure} $\Phi : M \to \rL^\infty(G/P)$ is a $\Gamma$-equivariant faithful normal ucp map. We say that $\Phi$ is \emph{invariant} if $\Phi(M) = \bfC 1$.
\end{definition}

We will simply say that $\Phi : M \to \rL^\infty(G/P)$ is a \emph{boundary structure} instead of a $\Gamma$-boundary structure when it is understood that $M$ is a $\Gamma$-von Neumann algebra. In this survey, we only deal with higher rank lattices in connected semisimple \emph{real}  Lie groups. In that setting, the notion of boundary structure is equivalent to the notion of stationary state. Indeed, fix a Furstenberg measure $\mu_\Gamma \in \Prob(\Gamma)$ so that $(G/P, \nu_P)$ is the $(\Gamma, \mu_\Gamma)$-Poisson boundary (see Theorem \ref{thm:poisson-Gamma}).
\begin{itemize}
\item If $\Phi : M \to \rL^\infty(G/P)$ is a boundary structure, then $\varphi \coloneqq \nu_P \circ \Phi \in M_\ast$ is a faithful normal $\mu_\Gamma$-stationary state on $M$. Moreover, if $\Phi$ is invariant, then $\varphi$ is $\Gamma$-invariant.
\item Conversely, let $\varphi \in M_\ast$ be any faithful normal $\mu_\Gamma$-stationary state on $M$. Define the $\Gamma$-equivariant faithful normal ucp map
$$\Phi : M \to \Har^\infty(\Gamma, \mu_\Gamma) : x \mapsto (\gamma \mapsto \varphi(\gamma^{-1} x)).$$
Since $\Har^\infty(\Gamma, \mu_\Gamma) \cong \rL^\infty(G/P, \nu_P)$ as $\Gamma$-operator systems, we may further regard $\Phi : M \to \rL^\infty(G/P)$ as a boundary structure such that $\varphi = \nu_P \circ \Phi$. If $\varphi$ is $\Gamma$-invariant, then $\Phi$ is invariant.
\end{itemize}

\begin{remark}
The notion of boundary structure was developed in \cite{BBHP20} to replace the notion of stationary state used in \cite{BH19} in order to deal with higher rank lattices in semisimple algebraic groups defined over \emph{arbitrary} local fields.
\end{remark}

It is useful to restrict boundary structures to separable $\rC^*$-subalgebras. Let $M$ be any $\Gamma$-von Neumann algebra with separable predual. A globally $\Gamma$-invariant separable ultraweakly dense $\rC^*$-subalgebra $A \subset M$ is called a \emph{separable model} for the action $\Gamma \curvearrowright M$. If $\Phi : M \to \rL^\infty(G/P)$ is a boundary structure, then $(\nu_P \circ \Phi)|_A \in \mathfrak S_{\mu_\Gamma}(A)$ and the restriction $\Phi|_A : A \to \rL^\infty(G/P)$ gives arise to the $\Gamma$-equivariant boundary map $\beta : G/P \to \mathfrak S(A) : b \mapsto \beta_b$ such that $\bary(\beta_\ast \nu_P) = (\nu_P \circ \Phi)|_A$, where
\begin{equation*}
\forall a \in A, \quad \Phi(a)(b) = \beta_b(a).
\end{equation*}

We present several examples of boundary structures.

\begin{example}[Boundary structure arising from unitary representations]\label{example:boundary-rep}
Let $\pi : \Gamma \to \mathscr U(\mathscr H_\pi)$ be any unitary representation and set $A \coloneqq \rC^*_\pi(\Gamma)$. Choose an extremal $\mu_\Gamma$-stationary state $\varphi \in \mathfrak S_{\mu_\Gamma}(A)$ and consider the GNS triple $(\pi_\varphi, \mathscr H_\varphi, \xi_\varphi)$. Denote by $\beta : G/P \to \mathfrak S(A) : b \mapsto \beta_b$ the $\Gamma$-equivariant boundary map associated with $\varphi \in \mathfrak S_{\mu_\Gamma}(A)$. By duality, we may consider the $\Gamma$-equivariant ucp map $\Phi : A \to \rL^\infty(G/P) : a \mapsto( b \mapsto \beta_b(a))$ which satisfies $\nu_P \circ \Phi = \varphi$. Set $M \coloneqq \pi_\varphi(A)\dpr = (\pi_\varphi \circ \pi)(\Gamma)\dpr$. By extremality, the conjugation action $\Ad (\pi_\varphi \circ \pi) : \Gamma \curvearrowright M$ is ergodic. Moreover, the $\Gamma$-equivariant ucp map
\begin{equation*}
 \pi_\varphi(A) \to \rL^\infty(G/P) : \pi_\varphi(a) \mapsto \Phi(a)
\end{equation*}
is well defined and extends to a boundary structure $\Phi : M \to \rL^\infty(G/P)$. We refer to \cite[Proof of Theorem A]{BH19} for further details.
\end{example}

\begin{example}[Boundary structure arising from characters]\label{example:boundary-characters} Let $\varphi : \Gamma \to \bfC$ be any extremal character. Simply denote by $(\pi, \mathscr H, \xi)$ the GNS triple associated with $\varphi \in \Char(\Gamma)$. Denote by $J : \mathscr H \to \mathscr H : \pi(\gamma)\xi \mapsto \pi(\gamma)^*\xi$ the canonical conjugation. Following \cite{Pe14}, define the \emph{noncommutative Poisson boundary} $\mathscr B$ as the von Neumann algebra of all $\nu_P$-equivalence classes of (essentially) bounded measurable functions $f : G/P \to \mathrm B(\mathscr H)$ satisfying $f(\gamma b) = \Ad(J \pi(\gamma) J)(f(b))$ for every $\gamma \in \Gamma$ and almost every $b \in G/P$. Observe that $\bfC 1 \otimes \pi(\Gamma)\dpr \subset \mathscr B$. Since $P$ is amenable, $\mathscr B$ is an amenable von Neumann algebra. By extremality, the conjugation action $\Ad(\pi) : \Gamma \curvearrowright \mathscr B$ is ergodic. Moreover, 
\begin{equation*}
\Phi : \mathscr B \to \rL^\infty(G/P) : f \mapsto (b \mapsto \langle f(b)\xi, \xi\rangle)
\end{equation*}
is a boundary structure. When $\varphi = \delta_e$ is the regular character, the noncommutative Poisson boundary $\mathscr B$ coincides with the group measure space von Neumann algebra $\rL(\Gamma \curvearrowright G/P)$ and the boundary structure $\Phi : \rL(\Gamma \curvearrowright G/P) \to \rL^\infty(G/P)$ is the canonical $\Gamma$-equivariant conditional expectation. We refer to \cite[Proof of Theorem C]{BH19} for further details.
\end{example}

\begin{example}[Boundary structure arising from topological dynamics]\label{example:boundary-top}
Let $\Gamma \curvearrowright X$ be any minimal action on a compact metrizable space. Choose an extremal $\mu_\Gamma$-stationary Borel probability measure $\nu \in \Prob_{\mu_\Gamma}(X)$. By minimality, we have $\supp(\nu) = X$. Denote by $\beta : G/P \to \Prob(X) : b \mapsto \beta_b$ the $\Gamma$-equivariant boundary map associated with $\nu \in \Prob_{\mu_\Gamma}(X)$. By duality, we may consider the $\Gamma$-equivariant ucp map $\Phi : \rC(X) \to \rL^\infty(G/P) : f \mapsto( b \mapsto \beta_b(f))$ which satisfies $\nu_P \circ \Phi = \nu$. By extremality, the nonsingular action $\Gamma \curvearrowright (X, \nu)$ is ergodic. Moreover, $\Phi : \rC(X) \to \rL^\infty(G/P)$
extends to  a boundary structure $\Phi : \rL^\infty(X, \nu) \to \rL^\infty(G/P)$.
\end{example}

The notion of boundary structure is well adapted to induction. Indeed, let $\Phi : M \to \rL^\infty(G/P)$ be any $\Gamma$-boundary structure. Denote by $\widehat M \coloneqq \Ind_\Gamma^G(M) \cong \rL^\infty(G/\Gamma) \ovt M$ the induced $G$-von Neumann algebra. Since $G/P$ is a $G$-space, we have $\Ind_\Gamma^G(\rL^\infty(G/P)) \cong \rL^\infty(G/\Gamma) \ovt \rL^\infty(G/P)$, where $G \curvearrowright G/\Gamma \times G/P$ acts diagonally. Denote by $\nu_{\Gamma} \in \Prob(G/\Gamma)$ the unique $G$-invariant Borel probability measure. Then the map $\widehat \Phi \coloneqq \nu_\Gamma \otimes \Phi : \widehat M \to \rL^\infty(G/P)$ is a $G$-equivariant faithful normal ucp map. We then refer to $\widehat \Phi$ as the \emph{induced} $G$-\emph{boundary structure}. Note that $\Phi$ is invariant if and only if $\widehat \Phi$ is invariant. 

This framework provides a more conceptual approach to the \emph{stationary induction} considered in \cite[Section 4]{BH19}. Let $\mu_G \in \Prob(G)$ be any $K$-invariant admissible Borel probability measure. Let $\varphi $ be any faithful normal $\mu_\Gamma$-stationary state on $M$ and define the corresponding $\Gamma$-boundary structure $\Phi : M \to \rL^\infty(G/P)$ such that $\nu_P \circ \Phi = \varphi$. Consider the induced $G$-boundary structure $\widehat \Phi : \widehat M \to \rL^\infty(G/P)$. Then $\widehat \varphi \coloneqq \nu_P \circ \widehat \Phi$ is a faithful normal $\mu_G$-stationary state on $\widehat M$. Moreover, $\varphi$ is $\Gamma$-invariant if and only if $\widehat \varphi$ is $G$-invariant.

\subsection{The dynamical dichotomy theorem for boundary structures}

Let $A$ be any separable $\rC^*$-algebra. We say that $\phi, \psi \in \mathfrak S(A)$ are \emph{pairwise singular} and write $\phi \perp \psi$ if there exists a sequence $(a_k)_k$ in $A$ such that $0 \leq a_k \leq 1$ for every $k \in \bfN$ and for which $\lim_k \phi(a_k) = 0 = \lim_k \psi(1 - a_k)$. This notion naturally extends the notion of pairwise singularity of Borel probability measures on metrizable compact spaces. Observe that for any unital $\rC^*$-subalgebra $B \subset A$ and any states $\phi, \psi \in \mathfrak S(A)$, if $\phi|_B \perp \psi|_B$, then $\phi \perp \psi$. We introduce the following terminology.

\begin{definition}[\cite{BBHP20}]
Let $\Gamma < G$ be any higher rank lattice. Let $M$ be any $\Gamma$-von Neumann algebra with separable predual and $\Phi : M \to \rL^\infty(G/P)$ any boundary structure. We say that $\Phi$ is \emph{singular} if there exists a separable model $A \subset M$ for the action $\Gamma \curvearrowright M$ such that the corresponding $\Gamma$-equivariant boundary map $\beta : G/P \to \mathfrak S(A) : b \mapsto \beta_b$ satisfies the following property:
\begin{equation}\label{eq:singularity}
\text{For every } \gamma \in \Gamma \setminus \mathscr Z(\Gamma), \text{ for almost every } b \in G/P, \quad \beta_{\gamma b} \perp \beta_b.
\end{equation}
\end{definition}

The notion of singularity for boundary structures is quite robust. If $\Phi : M \to \rL^\infty(G/P)$ is singular, then for {\em every} separable model $A \subset M$, the corresponding  $\Gamma$-equivariant boundary map $\beta : G/P \to \mathfrak S(A) : b \mapsto \beta_b$ satisfies \eqref{eq:singularity} (see \cite[Proposition 4.10]{BBHP20}). This implies the following useful fact. If $M_0 \subset M$ is a $\Gamma$-invariant von Neumann subalgebra and if the restriction $\Phi |_{M_0} : M_0 \to \rL^\infty(G/P)$ is singular, then $\Phi$ is singular as well.

In case the action $\Gamma \curvearrowright M$ is given by conjugation, singular boundary structures enjoy the following useful vanishing property.

\begin{proposition}[\cite{BBHP20}]\label{prop:dynamical-singular}
Let $M$ be any von Neumann algebra with separable predual and $\pi : \Gamma \to \mathscr U(M)$ any unitary representation. Consider the conjugation action $\Ad(\pi) : \Gamma \curvearrowright M$. Let $\Phi : M \to \rL^\infty(G/P)$ be any singular boundary structure. Then for every $\gamma \in \Gamma \setminus \mathscr Z(\Gamma)$, we have $\Phi(\pi(\gamma)) = 0$.
\end{proposition}

\begin{proof}
The proof is similar to \cite[Lemma 2.2]{HK17}. We may choose a separable model $A \subset M$ for the conjugation action $\Ad(\pi) : \Gamma \curvearrowright M$ such that $\pi(\Gamma) \subset A$. Denote by $\beta : G/P \to \mathfrak S(A)$ the $\Gamma$-equivariant boundary map arising from $\Phi|_A$. Let $\gamma \in \Gamma \setminus \mathscr Z(\Gamma)$ be any element. Choose a conull measurable subset $Y \subset G/P$ such that for every $b \in Y$, we have $\beta_{\gamma b} = \gamma \beta_b= \beta_b \circ \Ad(\pi(\gamma)^*)$ and $\beta_{\gamma b} \perp \beta_b$. Let $b \in Y$ be any point and choose a sequence $(a_k)_k$ in $A$ such that $0 \leq a_k \leq 1$ for every $k \in \bfN$ and for which $\lim_k \beta_{\gamma b}(a_k) = 0 = \lim_k \beta_b(1 - a_k)$. Then Cauchy--Schwarz inequality implies that
\begin{align*}
|\beta_b((1 - a_k)\pi(\gamma) )| &= |\beta_b((1 - a_k)^{1/2} \cdot (1 - a_k)^{1/2}\pi(\gamma) )| \\
&\leq \beta_b(1 - a_k)^{1/2} \to 0.
\end{align*}
Likewise, Cauchy--Schwarz inequality implies that
\begin{align*}
|\beta_{b}(a_k\pi(\gamma) )| &= |\beta_{\gamma b}(\pi(\gamma) a_k^{1/2} \cdot a_k^{1/2})| \\
&\leq \beta_{\gamma b}(a_k)^{1/2} \to 0.
\end{align*}
Then $\beta_b(\pi(\gamma)) = \beta_b((1 - a_k)\pi(\gamma) ) + \beta_b(a_k\pi(\gamma) ) \to 0$ and so $\beta_b(\pi(\gamma)) = 0$. Since this holds true for every $b \in Y$, it follows that $\Phi(\pi(\gamma)) = 0$.
\end{proof}

As we mentioned in the Introduction, the following dynamical dichotomy theorem \emph{invariant vs.\ singular} for boundary structures is the key novelty in our operator algebraic framework.

\begin{theorem}[\cite{BH19, BBHP20}]\label{thm:dynamical}
Let $M$ be any ergodic $\Gamma$-von Neumann algebra with separable predual and $\Phi : M \to \rL^\infty(G/P)$ any boundary structure. Then $\Phi$ is either invariant or singular.
\end{theorem}

The proof of Theorem \ref{thm:dynamical} depends heavily upon whether the ambient connected semisimple real Lie group $G$ is simple or not.

\textbf{In case $G$ is simple}, let us explain why Theorem \ref{thm:NCNZ} implies Theorem \ref{thm:dynamical}. Let $\Phi : M \to \rL^\infty(G/P)$ be any non-invariant boundary structure. By Theorem \ref{thm:NCNZ}, there exist a proper parabolic subgroup $P < Q < G$ and a $\Gamma$-equivariant unital normal embedding $\iota : \rL^\infty(G/Q) \hookrightarrow M$ such that $\Phi \circ \iota = p_Q^\ast$. Set $M_0 \coloneqq \iota(\rL^\infty(G/Q)) \subset M$. Then $A \coloneqq \iota(\rC(G/Q)) \subset M_0$ is a separable model for the action $\Gamma \curvearrowright M_0$ and the $\Gamma$-equivariant boundary map corresponding to $\Phi|_A$ is exactly $\delta \circ p_Q : G/P \to \Prob(G/Q) : gP \mapsto \delta_{gQ}$. Since any element $\gamma \in \Gamma \setminus \mathscr Z(\Gamma)$ acts (essentially) freely on $G/Q$ (see e.g.\! \cite[Lemma 6.2]{BH19}), it follows that the restriction $\Phi|_{M_0} : M_0 \to \rL^\infty(G/P)$ is singular. Thus, $\Phi$ is singular.

\textbf{In case $G$ is not simple}, we have to use a different approach. Following \cite{BBHP20}, we outline the main steps of the proof of Theorem \ref{thm:dynamical} in the particular case where $G = G_1 \times G_2$ with $G_1$ and $G_2$ noncompact connected simple Lie groups with finite center. This particular case already contains all the main conceptual difficulties. Let $i = 1, 2$. Denote by $p_i : G \to G_i$ the canonical factor map. Denote by $P_i < G_i$ a minimal parabolic subgroup and set $P = P_1 \times P_2 <G$. By Theorem \ref{thm:poisson-G}, $(G_i, \nu_{P_i})$ is the $(G_i, \mu_i)$-Poisson boundary with respect to appropriate Borel probability measures $\mu_i \in \Prob(G_i)$ and $\nu_{P_i} \in \Prob(G_i/P_i)$. Let $\Phi : M \to \rL^\infty(G/P)$ be any non-invariant boundary structure. Our goal is to show that $\Phi$ is singular.

\textbf{Step 1: Induction.} Denote by $\widehat \Phi : \widehat M \to \rL^\infty(G/P)$ the induced $G$-boundary structure. Since $\Phi$ is not invariant, $\widehat \Phi$ is not invariant either, that is, $\widehat \Phi(\widehat M) \neq \bfC 1$.

\textbf{Step 2: Reduction to the von Neumann algebra of $G_1$-continuous elements.} Exploiting the product structure $G = G_1 \times G_2$ and up to permuting the indices, we show that the restriction of $\widehat \Phi$ to the $G_2$-fixed point von Neumann subalgebra $\widehat M^{G_2} \subset \widehat M$ still satisfies $\Phi(\widehat M^{G_2}) \neq \bfC 1$ and moreover $\widehat \Phi(\widehat M^{G_2}) \subset \rL^\infty(G_1/P_1)$. Exploiting that $\Gamma < G_1 \times G_2$ is irreducible and that $(G_i, \nu_{P_i})$ is the $(G_i, \mu_i)$-Poisson boundary, $i = 1, 2$, we show that the $G_1$-von Neumann algebra $\widehat M^{G_2}$ is $\Gamma$-isomorphic to the $\Gamma$-von Neumann subalgebra $M_1 \subset M$ of all elements $x \in M$ for which the action map $\Gamma \to M : \gamma \mapsto \sigma_\gamma(x)$ extends continuously to $G_1$. We say that $M_1 \subset M$ is the von Neumann subalgebra of $G_1$-\emph{continuous elements}. Moreover, under the identification $\widehat M^{G_2} = M_1$, we naturally have the identification $\widehat \Phi|_{\widehat M^{G_2}} = \Phi|_{M_1}$. Then $M_1$ is a $G_1$-ergodic von Neumann algebra and $\Phi|_{M_1} : M_1 \to \rL^\infty(G_1/P_1)$ is a $G_1$-boundary structure such that $\Phi(M_1) \neq \bfC 1$. Since $\Phi|_{M_1}$ is the restriction to $M_1$ of the $\Gamma$-boundary structure $\Phi$, it suffices to show that $\Phi|_{M_1}$ is singular.

\textbf{Step 3: Singularity of $\Phi|_{M_1}$.} We may choose a separable model $A_1 \subset M_1$ for the continuous action $G_1 \curvearrowright M_1$. Since $G_1 \curvearrowright G_1/P_1$ is transitive, $\Phi|_{A_1} : A_1 \to \rL^\infty(G_1/P_1)$ gives rise to a $G_1$-equivariant continuous boundary map $\beta : G_1/P_1 \to \mathfrak S(A_1) : b \mapsto \beta_b$. Since the action $G_1 \curvearrowright M_1$ is ergodic, the $P_1$-invariant state $\psi \coloneqq \beta_{P_1} \in \mathfrak S(A_1)$ is extremal among $P_1$-invariant states. We then show that for every $g \in G_1$, either $g\psi \perp \psi$ or $g \psi = \psi$. Since $\psi$ is not $G_1$-invariant on $A_1$, the stabilizer $Q_1 = \Stab_{G_1}(\psi)$ is a proper parabolic subgroup such that $P_1 < Q_1$. Since any element $g \in G_1 \setminus \mathscr Z(G_1)$ acts (essentially) freely on $G_1/Q_1$ and since $p_1(\Gamma \setminus \mathscr Z(\Gamma)) \subset G_1 \setminus \mathscr Z(G_1)$, it follows that the restriction $\Phi|_{M_1}$ is singular. Thus, $\Phi$ is singular.

\subsection{Outline of the proof of Theorem \ref{thm:NCNZ}}

Following \cite{BH19}, we outline the main steps of the proof of Theorem \ref{thm:NCNZ}. We may assume that $G$ is a connected simple real Lie group with trivial center and real rank $\rk_{\bfR}(G) \geq 2$. Recall that $G$ admits an Iwasawa decomposition $G = KAV$ where $K < G$ is a maximal compact subgroup, $A < G$ is a Cartan subgroup and $V < G$ is a unipotent subgroup. Set $P = LV$ where $L = \mathscr Z_G(A)$ and note that $P < G$ is a minimal parabolic subgroup. Likewise, write $\overline P = L \overline V$ for the opposite minimal parabolic subgroup. Note that $G = \langle P, \overline V\rangle$ and the map $\overline V \to G/P : \overline v \mapsto \overline v P$ defines a measurable isomorphism. Let $\Phi : M \to \rL^\infty(G/P)$ be any non-invariant boundary structure. Our goal is to show that there exist a proper parabolic subgroup $P < Q < G$ and a $\Gamma$-equivariant unital normal embedding $\iota : \rL^\infty(G/Q) \hookrightarrow M$.

\textbf{Step 1: Induction.} Exactly as in the proof of Theorem \ref{thm:dynamical}, denote by $\widehat \Phi : \widehat M \to \rL^\infty(G/P)$ the induced $G$-boundary structure which satisfies $\widehat \Phi(\widehat M) \neq \bfC 1$. We may choose a separable model $A \subset \widehat M$ for the continuous action $G \curvearrowright \widehat M$. Since $G \curvearrowright G/P$ is transitive, $\widehat \Phi|_{A} : A \to \rL^\infty(G/P)$ gives rise to a $G$-equivariant continuous boundary map $\beta : G/P \to \mathfrak S(A) : b \mapsto \beta_b$. Denote by $\psi \coloneqq \beta_P \in \mathfrak S(A)$ the corresponding $P$-invariant state, consider the GNS representation $\pi_\psi :A \to \mathrm B( \mathscr H_\psi)$ and set $N \coloneqq \pi_\psi(A)\dpr$. Then $N$ is a $P$-von Neumann algebra and we may consider the induced $G$-von Neumann algebra $\Ind_P^G(N)$. Moreover, we may regard $\widehat M \subset \Ind_P^G(N)$ as a $G$-von Neumann subalgebra. We point out that the normal state $\psi \in N_\ast$ need not be faithful on $N$. Since we only give a sketch of the proof, we will assume that $\psi$ is faithful on $N$. We refer to \cite[Theorem 5.1]{BH19} for the general proof.

\textbf{Step 2: Construction of a well behaved von Neumann subalgebra.} In this step, we build upon Nevo--Zimmer's proof of \cite[Theorem 1]{NZ00}. Since $\widehat \Phi(\widehat M) \neq \bfC 1$, the $P$-invariant state $\psi \in \mathfrak S(A)$ is not $G$-invariant whence not $\overline V$-invariant. Using the real rank assumption $\rk_{\bfR}(G) \geq 2$, there is a strict intermediate parabolic subgroup $P < P_0 < G$ with Levi decomposition $P_0 = L_0 V_0$, where $L_0 = \mathscr Z_G(A_0)$, $A_0 < A$, and $V_0 < V$, such that $\psi \in \mathfrak S(A)$ is not $\overline V_0$-invariant. Choose a nontrivial element $s \in A_0$ so that $s$ (resp.\! $s^{-1}$) acts by conjugation as a contracting automorphism of $V_0$ (resp.\! $\overline V_0$). By Mautner phenomenon, any $s$-fixed element in $N$ is necessarily $V_0$-fixed. Since the subgroup $\langle s, V_0\rangle$ is normal in $P$, it follows that $N^s \subset N$ is a $P$-invariant von Neumann subalgebra. Using these assumptions, we show that $\mathscr M_0 = \widehat M \cap \Ind_P^G(N^s)$ is a $G$-von Neumann subalgebra such that $\widehat \Phi(\mathscr M_0) \neq \bfC 1$.

\textbf{Step 3: From noncommutative to commutative.} We reached the point where we can no longer rely on Nevo--Zimmer's argument \cite{NZ00}. Indeed, $\mathscr M_0$ is not commutative and so we cannot use the Gauss map trick from \cite[Section 3]{NZ00} to conclude. We adopt the following new strategy. Using the induction in two steps, we may write $\Ind_P^G (N^s) = \Ind_{P_0}^G (\Ind_P^{P_0}(N^s)) \cong \rL^\infty(\overline V_0, \Ind_P^{P_0}(N^s))$. Regarding $\mathscr M_0 \subset \rL^\infty(\overline V_0, \Ind_P^{P_0}(N^s))$ as a $G$-von Neumann subalgebra, denote by $\mathscr N_0 \subset \Ind_P^{P_0}(N^s)$ the von Neumann subalgebra generated by the essential values of all the elements $f \in \mathscr A_0$, where $\mathscr A_0 \subset \mathscr M_0$ is an ultraweakly dense separable $\rC^*$-subalgebra. On the one hand, by construction, we have $\mathscr M_0 \subset \rL^\infty(\overline V_0, \mathscr N_0) \cong \rL^\infty(\overline V_0) \ovt \mathscr N_0$. On the other hand, exploiting that $s^{-1}$ acts by conjugation as a contracting automorphism of $\overline V_0$, we show that $\bfC 1 \ovt \mathscr N_0 \subset \mathscr M_0$. Therefore, we have the inclusions $\bfC 1 \ovt \mathscr N_0 \subset \mathscr M_0 \subset \rL^\infty(\overline V_0) \ovt \mathscr N_0$. By considering the centers, we also have $\bfC 1 \ovt \mathscr Z(\mathscr N_0) \subset \mathscr Z( \mathscr M_0 ) \subset \rL^\infty(\overline V_0) \ovt  \mathscr Z(\mathscr N_0)$. 
Since $G \curvearrowright \mathscr M_0$ is faithful and since $s$ acts identically on $\bfC 1 \ovt \mathscr N_0$, we have $\bfC1 \ovt \mathscr N_0 \subsetneq \mathscr M_0$. Exploiting Ge--Kadison's splitting theorem \cite{GK95}, we show that $\bfC 1 \ovt \mathscr Z(\mathscr N_0) \neq \mathscr Z( \mathscr M_0 )$. This further implies that $\Phi(\mathscr Z(\mathscr M_0)) \neq \bfC 1$. We may now apply Nevo--Zimmer's result \cite[Theorem 1]{NZ00} to the commutative $G$-von Neumann algebra $\mathscr Z(\mathscr M_0)$ to obtain a proper parabolic subgroup $P < Q < G$ and a $G$-equivariant unital normal embedding $\iota : \rL^\infty(G/Q) \hookrightarrow \mathscr Z(\mathscr M_0) \subset \widehat M$.

\textbf{Step 4: Back to the lattice.} We use the simple argument given in \cite{BBH21}. Since the action $G \curvearrowright \rC(G/Q)$ is $\|\cdot\|$-continuous, by $G$-equivariance, it follows that $\iota(\rC(G/Q))$ is contained in the $\rC^*$-subalgebra $\rC^*_b(G, M)^\Gamma \subset \widehat M$ of all $\Gamma$-equivariant bounded continuous functions $f : G \to M$. Consider the evaluation $\ast$-homomorphism $\mathsf{ev} : \rC^*_b(G, M)^\Gamma \to M : f \mapsto f(e)$. Then $\mathsf{ev} \circ \iota : \rC(G/Q) \to M$ is a $\Gamma$-equivariant $\ast$-homomorphism which extends uniquely to a $\Gamma$-equivariant unital normal embedding $\rL^\infty(G/Q) \hookrightarrow M$ thanks to \eqref{eq:uniqueness}.

\section{Proofs of the main results}\label{section:main}

In this section, following \cite{BH19, BBHP20}, we explain how to use the dynamical dichotomy Theorem \ref{thm:dynamical} to prove the main results stated in the Introduction. We fix a higher rank lattice $\Gamma < G$ and a Furstenberg measure $\mu_\Gamma \in \Prob(\Gamma)$ so that $(G/P, \nu_P)$ is the $(\Gamma, \mu_\Gamma)$-Poisson boundary (see Theorem \ref{thm:poisson-Gamma}). Since the proofs of Theorems \ref{thm:main} and \ref{thm:dynamics} are similar, we only give the proofs of Theorems \ref{thm:main}, \ref{thm:characters} and Corollary \ref{cor:C*}.

\begin{proof}[Proof of Theorem \ref{thm:main}]
Denote by $\pi : \Gamma \to \mathscr U(\mathscr H_\pi)$ the \emph{universal} unitary representation, meaning that $\pi$ is equal to the orthogonal direct sum of all cyclic unitary representations of $\Gamma$. Then $A \coloneqq \rC^*_\pi(\Gamma)$ coincides with the full $\rC^*$-algebra $\rC^*(\Gamma)$ and we may use the identification $\mathfrak S(A) = \mathscr P(\Gamma)$. Let $\mathscr C \subset \mathfrak S(A)$ be any nonempty $\Gamma$-invariant weak-$\ast$ compact convex subset. We claim that any $\mu_\Gamma$-stationary state $\varphi \in \mathscr C$ is $\Gamma$-invariant. More generally, we claim that any $\mu_\Gamma$-stationary state on $A$ is $\Gamma$-invariant. By Krein--Milman theorem, it suffices to show that any extremal $\mu_\Gamma$-stationary $\varphi \in \mathfrak S(A)$ is $\Gamma$-invariant. Letting $M = \pi_\varphi(A)\dpr$, consider the boundary structure $\Phi : M \to \rL^\infty(G/P)$ such that $\nu_P \circ \Phi = \varphi$ as in Example \ref{example:boundary-rep}. By Theorem \ref{thm:dynamical}, $\Phi$ is either invariant or singular. If $\Phi$ is invariant, then $\varphi \in \Char(\Gamma)$. If $\Phi$ is singular, then Proposition \ref{prop:dynamical-singular} implies that for every $\gamma \in \Gamma \setminus \mathscr Z(\Gamma)$, we have $\Phi(\pi(\gamma)) = 0$ and so $\varphi(\gamma) = 0$. Then $\varphi$ is supported on $\mathscr Z(\Gamma)$ and so $\varphi \in \Char(\Gamma)$.
\end{proof}

\begin{proof}[Proof of Theorem \ref{thm:characters}]
Let $\varphi \in \Char(\Gamma)$ be any character. We may assume that $\varphi$ is an extremal character. Denote by $\mathscr B$ the noncommutative Poisson boundary and  consider the boundary structure $\Phi : \mathscr B \to \rL^\infty(G/P)$ as in Example \ref{example:boundary-characters}. By Theorem \ref{thm:dynamical}, $\Phi$ is either invariant or singular.  If $\Phi$ is invariant, then for every $f \in \mathscr B$, the function $\Phi(f) : G/P \to \bfC : b \mapsto \langle f(b)\xi, \xi\rangle$ is (essentially) constant. Since  $\bfC1 \otimes \pi_\varphi(\Gamma)\dpr \subset \mathscr B$ and since the linear span of $\pi_\varphi(\Gamma)\xi_\varphi$ is dense in $ \mathscr H_\varphi$, it easily follows that every $f \in \mathscr B$ is (essentially) constant as a function $G/P \to \mathrm B(\mathscr H_\varphi)$. This further implies that $\mathscr B = \bfC 1 \otimes \pi_\varphi(\Gamma)\dpr$ and so $\pi_\varphi(\Gamma)\dpr$ is amenable. This further implies that $\pi_\varphi$ is amenable. If $\Phi$ is singular, then  for every $\gamma \in \Gamma \setminus \mathscr Z(\Gamma)$ we have $\Phi(\pi(\gamma)) = 0$ and so $\varphi (\gamma) = 0$. Thus, $\varphi$ is supported on $\mathscr Z(\Gamma)$.

If $G$ has property (T), then $\Gamma$ has property (T). If $\varphi$ is not supported on $\mathscr Z(\Gamma)$, then $\pi_\varphi$ is amenable and so $\pi_\varphi$ necessarily contains a finite dimensional subrepresentation. In the more general case where $G$ has a simple factor with property (T), we refer to the proof of \cite[Proposition 7.5]{BBHP20}.
\end{proof}

\begin{proof}[Proof of Corollary \ref{cor:C*}]
Let $\pi : \Gamma \to \mathscr U(\mathscr H_\pi)$ be any unitary representation and set $A \coloneqq \rC^*_\pi(\Gamma)$. Regarding $\mathfrak S(A) \subset \mathscr P(\Gamma)$ as a $\Gamma$-invariant weak-$\ast$ compact convex subset, Theorem \ref{thm:main} implies that the $\rC^*$-algebra $A$ admits a trace. 

Assume moreover that $G$ has trivial center and that $\pi$ is not amenable. Let $\varphi \in \mathfrak S(A)$ be any trace. Regarding $\varphi \in \Char(\Gamma)$ as a character, the GNS representation $\pi_\varphi$ is weakly contained in $\pi$ and so $\pi_\varphi$ is not amenable. Theorem \ref{thm:characters} implies that $\varphi = \delta_e$. Then $\pi_\varphi = \lambda$ is the left regular representation and $\lambda$ is weakly contained in $\pi$. Denote by $\Theta : \rC^*_\pi(\Gamma) \to \rC^*_\lambda(\Gamma)$ the unique $\ast$-homomorphism such that $\Theta(\pi(\gamma)) = \lambda(\gamma)$ for every $\gamma \in \Gamma$. Then $\tau_\Gamma \circ \Theta$ is the unique trace on $A = \rC^*_\pi(\Gamma)$. Let $J \lhd A$ be any proper ideal and define $\rho : A \to A/J$. Then the unitary representation $\rho \circ \pi$ is weakly contained in $\pi$ and so $\rho \circ \pi$ is not amenable. The previous reasoning implies that $\lambda$ is weakly contained in $\rho \circ \pi$ and so there is a $\ast$-homomorphism $\overline \Theta : A/J \to \rC^*_\lambda(\Gamma)$ such that $\Theta = \overline \Theta \circ \rho$. Then $J = \ker(\rho) \subset \ker(\Theta)$. Therefore, $\ker(\Theta) \lhd \rC^*_\pi(\Gamma)$ is the unique maximal proper ideal. 
\end{proof}

We now discuss open problems in relation with our main results. Let $(E, \|\cdot\|)$ be any separable Banach $\Gamma$-module and $\mathscr C \subset E^*$ any nonempty $\Gamma$-invariant weak-$\ast$ compact convex subset. Let $\mu \in \Prob(\Gamma)$ be any probability measure. By analogy with \cite{Fu98}, we say that the affine action $\Gamma \curvearrowright \mathscr C$ is $\mu$-\emph{stiff} if any $\mu$-stationary point is invariant. The proof of Theorem \ref{thm:main} shows that for any Furstenberg measure $\mu_\Gamma \in \Prob(\Gamma)$, the affine action $\Gamma \curvearrowright \mathscr P(\Gamma)$ is $\mu_\Gamma$-stiff. It is natural to ask whether the stiffness property holds for more general probability measures $\mu \in \Prob(\Gamma)$.

\begin{problem}\label{prob:stationary}
Let $\mu \in \Prob(\Gamma)$ be \emph{any} probability measure such that $\langle \supp(\mu)\rangle = \Gamma$. Is the action $\Gamma \curvearrowright \mathscr P(\Gamma)$ $\mu$-stiff? 
\end{problem}

Problem \ref{prob:stationary} requires a new strategy as we can no longer identify the $(\Gamma, \mu)$-Poisson boundary with $(G/P, \nu_P)$. We note that in case $\Gamma$ has trivial center, it is showed in \cite{HK17} that for any probability measure $\mu \in \Prob(\Gamma)$ such that $\langle \supp(\mu)\rangle = \Gamma$, the affine action $\Gamma \curvearrowright \mathfrak S(\rC^*_\lambda(\Gamma))$ is $\mu$-stiff and the canonical trace $\tau_\Gamma$ is the only $\mu$-stationary state on $\rC^*_\lambda(\Gamma)$. Problem \ref{prob:stationary} is particularly relevant in case $\Gamma \coloneqq \SL_d(\bfZ) < \SL_d(\bfR) \coloneqq G$ for $d \geq 3$. To draw a parallel with homogeneous dynamics, we point out that the affine action $\SL_d(\bfZ) \curvearrowright \Prob(\bfT^d)$ is $\mu$-stiff for every $d \geq 2$ and every probability measure $\mu \in \Prob(\SL_d(\bfZ))$ such that $\langle \supp(\mu)\rangle = \SL_d(\bfZ)$ \cite{BFLM09}. We refer the reader to \cite{BFLM09} and \cite{BQ09} for more general results regarding measure rigidity.

We say that a countable discrete group $\Lambda$ is \emph{character rigid} if for any extremal character $\varphi \in \Char(\Lambda)$, either $\varphi$ is supported on $\mathscr Z(\Lambda)$ or the GNS representation $\pi_\varphi$ is finite dimensional. Assuming that $G$ has a simple factor with property (T), Theorem \ref{thm:characters} says that $\Gamma$ is character rigid. It is showed in \cite{PT13} that $\SL_d(\bfZ[S^{-1}])$ is character rigid for every $d \geq 2$ and every nonempty set of primes $S$. In view of Margulis'\! normal subgroup theorem which holds for arbitrary higher rank lattices, the next problem is of fundamental importance.

\begin{problem}\label{prob:NCSZ}
Let $\Gamma < G$ be \emph{any} higher rank lattice. Is $\Gamma$ character rigid?
\end{problem}

Problem \ref{prob:NCSZ} is also discussed in \cite[Question 2.1]{Ge18} for characters arising from ergodic probability measure preserving actions $\Gamma \curvearrowright (X, \nu)$, in connection with Stuck--Zimmer's results \cite{SZ92}. To answer positively Problem \ref{prob:NCSZ}, it would suffice to prove for every extremal character $\varphi \in \Char(\Gamma)$ that is not supported on $\mathscr Z(\Gamma)$, the tracial GNS factor $\pi_\varphi(\Gamma)\dpr$ has property (T) in the sense of Connes--Jones \cite{CJ83}. This would correspond to the \emph{property} (T) \emph{half} in Margulis'\! normal subgroup theorem.

\section{Noncommutative factor theorem and  Connes'\ rigidity conjecture}\label{section:NCFT}

Connes \cite{Co80} obtained the first rigidity result in von Neumann algebras by showing that for any icc (infinite conjugacy classes) countable discrete group $\Lambda$ with property (T), the type $\mathrm{II}_1$ factor $M = \rL(\Lambda)$ has countable outer automorphism group $\Out(M)$ and countable fundamental group $\mathscr F(M)$\footnote{The \emph{fundamental group} $\mathscr F(M)$ of a type $\mathrm{II}_1$ factor $M$ is defined as the subgroup of $\bfR^*_+$ that consists of all $\frac{\tau(p)}{\tau(q)}$, where $p, q \in M$ are projections for which $pMp \cong q M q$.}. This result prompted Connes to state the following bold conjecture (see \cite[Problem V.B.1]{Co94}).

\begin{connes}
Let $\Lambda_1$ and $\Lambda_2$ be any icc countable discrete groups with property (T) such that $\rL(\Lambda_1) \cong \rL(\Lambda_2)$. Show that $\Lambda_1 \cong \Lambda_2$.
\end{connes}

We say that a discrete group $\Lambda$ is W$^*$-\emph{superrigid} if whenever $\Upsilon$ is another discrete group such that $\rL(\Lambda) \cong \rL(\Upsilon)$, we have $\Lambda \cong \Upsilon$. Using \cite{CJ83}, Connes'\! rigidity conjecture asks whether every icc countable discrete group $\Lambda$ with property (T) is W$^*$-superrigid. 

A first deep result towards Connes'\! rigidity conjecture was obtained by Cowling--Haagerup \cite{CH88} where they showed that for every $n \geq 2$ and every lattice $\Lambda$ in the rank one connected simple Lie group $\Sp(n, 1)/\{\pm 1\}$, the type $\mathrm {II}_1$ factor $\rL(\Lambda)$ retains the integer $n \geq 2$. For the last two decades,  Popa's \emph{deformation/rigidity} theory \cite{Po06} has led to tremendous progress regarding classification and rigidity results for group (resp.\ group measure space) von Neumann algebras. In particular, the first examples of W$^*$-superrigid groups were obtained by Ioana--Popa--Vaes \cite{IPV10}. The examples constructed in \cite{IPV10} are generalized wreath products groups and so they do not have property (T). It is still an open problem to find an example of a W$^*$-superrigid countable discrete group $\Lambda$ with property (T). For other recent results regarding classification and rigidity results for von Neumann algebras, we refer the reader to the surveys \cite{Va10, Io12, Io18}.

Connes'\! rigidity conjecture is particularly relevant for the class of higher rank lattices. In this context, celebrated strong rigidity results by Mostow and Margulis (see \cite{Mo73} and \cite[Chapter VI]{Ma91}) show that whenever $\Gamma_i < G_i$ is a higher rank lattice, where $G_i$ has trivial center, $i = 1, 2$, if $\Gamma_1 \cong \Gamma_2$, then $G_1 \cong G_2$. In view of the strong rigidity results by Mostow and Margulis, we state the following version of Connes'\! rigidity conjecture for higher rank lattices.

\begin{conjecture}
For every  $i = 1, 2$, let $G_i$ be any connected simple real Lie group with trivial center and real rank $\rk_{\bfR}(G_i) \geq 2$ and $\Gamma_i < G_i$ any lattice. If $\rL(\Gamma_1) \cong \rL(\Gamma_2)$, then $G_1 \cong G_2$.
\end{conjecture}

Popa's \emph{deformation/rigidity} theory cannot be used to tackle the above conjecture because higher rank lattices are somehow ``too rigid''. As suggested by Connes himself (see the discussion in \cite[Section 4]{Jo00}), it is natural to try and develop a strategy building upon the works of Furstenberg, Margulis and Zimmer. In what follows, we assume that $G$ is a connected simple real Lie group with trivial center and real rank $\rk_{\bfR}(G) \geq 2$. We fix a minimal parabolic subgroup $P < G$. Let $\Gamma < G$ be any lattice (e.g.\ $\Gamma \coloneqq \PSL_d(\bfZ) < \PSL_d(\bfR) \coloneqq G$ for $d \geq 3$). Then $\Gamma$ is icc and the group von Neumann algebra $\rL(\Gamma)$ is a type $\mathrm{II}_1$ factor. Moreover, the nonsingular action $\Gamma \curvearrowright  (G/P, \nu_P)$ is (essentially) free and ergodic and the corresponding von Neumann factor $\rL(\Gamma \curvearrowright G/P)$ is amenable and of type $\mathrm{III}_1$ (see e.g.\ \cite[Proposition 4.7]{BN11}). We now give the proof of Corollary \ref{thm:NCFT} by combining Theorem \ref{thm:NCNZ} with Suzuki's results \cite{Su18}.

\begin{proof}[Proof of Corollary \ref{thm:NCFT}]
Denote by $\rE : \rL(\Gamma \curvearrowright G/P) \to \rL^\infty(G/P)$ the canonical $\Gamma$-equivariant conditional expectation. Let  $\rL(\Gamma) \subset M \subset \rL(\Gamma \curvearrowright G/P)$ be any intermediate von Neumann subalgebra and consider the boundary structure $\Phi = \rE|_M : M \to \rL^\infty(G/P)$. By Theorem \ref{thm:NCNZ}, there are two cases to consider.

Firstly, assume that $\Phi$ is invariant. Following Examples \ref{ex:vN}(2), we simply denote by $u_\gamma \in \rL(\Gamma) \subset \rL(\Gamma \curvearrowright G/P)$ the canonical unitaries implementing the action $\Gamma \curvearrowright G/P$. For every $x \in M$, write $x = \sum_{\gamma \in \Gamma} x_\gamma  u_\gamma$ for its Fourier expansion, where $x_\gamma = \rE(x u_\gamma^*)$ for every $\gamma \in \Gamma$. Since $\rE|_M = \Phi$ is invariant and since $\rL(\Gamma) \subset M$, it follows that $x_\gamma \in \bfC 1$ for every $\gamma \in \Gamma$ and every $x \in M$. This implies that $M = \rL(\Gamma)$ (see e.g.\ \cite[Lemma 6.8]{AHHM18}).

Secondly, assume that $\Phi$ is not invariant. Then there exist a proper parabolic subgroup $P < Q < G$  and a $\Gamma$-equivariant unital normal embedding $\iota : \rL^\infty(G/Q) \hookrightarrow M$ such that $\rE \circ \iota = p_Q^\ast$. This further implies that $\rL(\Gamma \curvearrowright G/Q) = \rL(\Gamma) \vee \rL^\infty(G/Q) \subset M$. Since the nonsingular action $\Gamma \curvearrowright (G/Q, \nu_Q)$ is (essentially) free (see e.g.\ \cite[Lemma 6.2]{BH19}), a combination of \cite[Theorem 3.6]{Su18} and \cite[Theorem IV.2.11]{Ma91} implies that there exists a parabolic subgroup $P < R < Q$ such that $M = \rL(\Gamma \curvearrowright G/R)$.
 \end{proof}

It is well known that there are exactly $2^{\rk_{\bfR}(G)}$ intermediate parabolic subgroups $P <  Q < G$. Thus, Corollary \ref{thm:NCFT} implies that there are exactly $2^{\rk_{\bfR}(G)}$ intermediate von Neumann subalgebras $\rL(\Gamma) \subset M \subset \rL(\Gamma \curvearrowright G/P)$. In particular, the inclusion $\rL(\Gamma) \subset \rL(\Gamma \curvearrowright G/P)$ retains the real rank $\rk_{\bfR}(G)$. We believe this observation could be useful to tackle Connes' rigidity conjecture and to show that the group von Neumann algebra $\rL(\Gamma)$ retains the real rank $\rk_{\bfR}(G)$.


\end{document}